\documentclass{cimart}

\usepackage{dutchcal}
\usepackage{float}
\usepackage{mathtools}
\usepackage{xcolor}

\newcommand{\wei}{\mathcal w}
\newcommand{\Wei}{\mathcal W}

\title{
    Summing the sum of digits
    }

\author{
    Jean-Paul Allouche and Manon Stipulanti
    }

\authorinfo[
    J.-P. Allouche]{
    CNRS, IMJ-PRG, Sorbonne, Paris, France}{%
    jean-paul.allouche@imj-prg.fr
    }

\authorinfo[
    M. Stipulanti]{
    FNRS, Department of Mathematics, University of Liège, Belgium}{%
    m.stipulanti@uliege.be
    }

\abstract{%
    We revisit and generalize inequalities for the summatory function
    of the sum of digits in a given integer base. We prove that several
    known results can be deduced from a theorem in a 2023 paper
    by Mohanty, Greenbury, Sarkany, Narayanan, Dingle, Ahnert, and 
    Louis, whose primary scope is the maximum mutational robustness 
    in genotype-phenotype maps.

    \bigskip
    
    \begin{flushright}
        \it We dedicate this work to Christiane Frougny \\
        on the occasion of her 75th birthday.
    \end{flushright}
    }



\VOLUME{33}
\YEAR{2025}
\ISSUE{2}
\NUMBER{2}
\DOI{https://doi.org/10.46298/cm.12610}

\begin{document}

\section{Introduction}

Looking at the sum of digits of integers in a given base has
been the subject of numerous papers. In particular, the summatory
function of the sum of digits, i.e., the sum of all digits of
all integers up to some integer, has received much attention.  
Such ``sums of sums'' can be viewed through several prisms, two 
of them being to obtain (optimal) inequalities on the one
hand and asymptotic formulas on the other. For the latter
approach we only cite the study {\it par excellence}, namely
the 1975 paper of Delange~\cite{Delange}. The results about
these sums sometimes occur in unexpected domains. The most 
prominent of them is the link with fractal functions, in 
particular with the Takagi function (a continuous function that 
is nowhere differentiable~\cite{Takagi1903}) and the blancmange 
curve -- see the nice surveys of Lagarias~\cite{Lagarias2012} 
and Allaart and Kawamura~\cite{AK2011}.

The present paper will concentrate on inequalities satisfied 
by these sums of sums. Knowing the inspiring 
paper of Graham (see~\cite{Graham70,jones-torrence} and also~\cite{McIlroy,Allo21}), we first found the 2011 paper of Allaart~\cite{Alla11}. Then, we came across the paper~\cite{bio} (of course, many other papers would deserve to be cited, for instance~\cite{DrazinGriffith1952,LM16}, as well as~\cite{HKST2024} where the authors have an unexpected use of a lemma of Graham in~\cite{Graham70}).
In~\cite{bio}, the authors speak about the {\it maximum mutational robustness in genotype–phenotype 
maps}: that paper drew our attention because it contains the 
expressions ``blancmange-like curve'', ``Takagi function'', and 
``sums of digits''.
In particular, the authors of~\cite{bio} 
prove the following theorem (they indicate that this generalizes 
the case $b=2$ addressed in Graham's paper~\cite{Graham70}).

\begin{theorem}[in {\cite[Thm~5.1, pp.~12--13]{bio}}]
\label{th:bio}
Let $b$ be an integer $\geq 2$. For all integers $n\ge 0$, let $s_b(n)$ denote the sum of  
digits in the base-$b$ expansion of the integer $n$ and define
$S_b(n) := \sum_{1 \leq j \leq n-1} s_b(j)$.
Let $n_1, n_2, \dots, n_b$ be integers such that 
$0 \leq n_1 \leq n_2 \leq \dots \leq n_b$. Then the following 
inequality holds:
\begin{equation}\label{ineq:bio-base-b}
\sum_{i=1}^{b} S_b(n_i) 
+ \sum_{i=1}^{b-1} (b-i)n_i
\leq S_b \left(\sum_{i=1}^{b} n_i\right).
\end{equation}
\end{theorem}

The 2011 paper by Allaart~\cite{Alla11} somehow goes in a similar 
direction. Namely, Allaart proves the following 
(see~\cite[Ineq.~4]{Alla11}).

\begin{theorem}[in {\cite{Alla11}}]\label{th:allaart}
Let $p$ be a real number. For all integers $n$, let its binary
expansion be $n = \sum_{i=0}^{+\infty} d_i 2^i$ where $d_i\in
\{0,1\}$ for all $i\ge 0$ and $d_i=0$ for all large enough $i$ and define $\wei_p(n) := \sum_{i=0}^{+\infty} 2^{pi} d_i$ and its 
summatory function by $\Wei_p(n) := \sum_{m=0}^{n-1} \wei_p(m)$.
Then, for all $p \in [0, 1]$ and all integers $\ell \in [0, m]$, the following 
inequality holds:
\begin{equation}\label{eq:Allaart-2powerp}
    \Wei_p(m+\ell) + \Wei_p(m-\ell) - 2 \Wei_p(m) \le \ell^{p+1}.
\end{equation}
\end{theorem}

When reading the literature on the subject we have noted that 
the most recent papers do not always cite more ancient ones,
confirming a remark of Stolarsky~\cite[p.~719]{Stolarsky}:
{\em Whatever its mathematical virtues, the literature 
on sums of digital sums reflects a lack of communication 
between researchers}. As one might add, reasons for this could
be that there are a very large number of papers dealing with sums 
of digits, and many of them are not directly interested in these 
sums {\it per se}, but because they occur in seemingly unrelated 
questions.

In this paper, we will first answer a question of Allaart about 
the case $p=0$ in~\cite{Alla11}; see Theorem~\ref{th:allaart}
above. Then, in Section~\ref{sec:results} we will give a corollary 
and two generalizations of the result in~\cite{bio} 
(Theorem~\ref{th:bio} above): we will prove that several results 
that we found in the literature can be actually deduced from this 
corollary and these two generalizations. Finally, we will ask a 
few questions about possible sequels to this work. 

\section{A quick lemma that will be used several times}

In this short section we give an easy useful lemma (the first 
equality can be found, e.g., as~\cite[Lem.~7, p.~683]{Alla14}, 
or as~\cite[Ex.~3.11.5, p.~112]{AS}; the other equalities 
are immediate consequences of the first one).

\begin{lemma}\label{lemma:allaart}

\begin{itemize}
    \item[(i)] For all integers $b \geq 2$ and $n \geq 1$, we have 
    \[S_b(bn) = b S_b(n) + \dfrac{b(b-1)}{2} n.\]
    \item[(ii)] For all integers $b \geq 2$, $n \geq 1$, and $x \geq 0$, 
     we have
$$
S_b(b^x n) = b^x\, S_b(n) + \dfrac{b-1}{2} \, x\, b^x \, n 
\ \ \text{and} \ \ \ 
S_b(b^x) = \dfrac{b-1}{2}\, x \, b^x.  
$$
\end{itemize}
\end{lemma}

\section{Graham's result implies the case \texorpdfstring{$p=0$}{p=0} of Allaart's result}
\label{sec:Graham-Allaart}

As mentioned above, the result in the paper of
Graham~\cite{Graham70} (also see~\cite{jones-torrence} 
and~\cite{McIlroy}, and a less elegant but possibly more 
natural proof in~\cite{Allo21}) corresponds to the case 
$b=2$ of Theorem~\ref{th:bio}, namely (with the usual convention 
on empty sums):

\begin{theorem}[in~\cite{Graham70}]\label{th:graham}
For all integers $n_1,n_2$ with $0 \leq n_1 \le n_2$, we have
\begin{equation}\label{ineq:graham-base-2}
    S_2(n_1) + S_2(n_2) + n_1 \leq S_2(n_1+n_2).
\end{equation}
\end{theorem}

The author of~\cite{Alla11} writes (middle of page~690) 
about Inequality~\eqref{eq:Allaart-2powerp}: {\em ``It seems that 
for $0 < p < 1$ the inequality may be new. In fact, even 
for the case $p = 0$ the author has not been able to find a 
reference''.} In this section we indicate that the case $p=0$ 
appears, in a slightly disguised form, in the 1970 paper of 
Graham.

\begin{proof}
First, we note that, taking $b=2$ in Lemma~\ref{lemma:allaart} (i) 
above, we obtain, 
\begin{equation}\label{eq:S(2t)}
    S_2(2t) = 2 S_2(t) + t
\end{equation}
for all positive integers $t$.
Now let $m, \ell$ be two integers with $0 \leq \ell \leq m$.
Define $n_1 := m - \ell$ and $n_2 = m+ \ell$. 
Then $0 \leq n_1 \leq n_2$. Graham's theorem and Equality~\eqref{eq:S(2t)} yield
$$
S_2(m - \ell) + S_2(m+ \ell) + m - \ell \leq S_2(2m) 
= 2 S_2(m) + m,
$$
and hence
$$
S_2(m - \ell) + S_2(m+ \ell) - 2 S_2(m) \leq \ell,
$$
as desired.
\end{proof}

\begin{remark}
Having shown that Graham's inequality implies the case $p=0$ 
in Allaart's inequality, one can ask whether Allaart's 
inequality for $p=0$ gives back Graham's. For two integers $n_1, n_2$ 
with $0 \leq n_1 \leq n_2$, we want to prove that Inequality~\eqref{ineq:graham-base-2} holds. If $n_1$ and $n_2$ 
have the same parity, then we can define $m$ and $\ell$ by
$$
m := \frac{n_1 + n_2}{2} \ \ \ell := \frac{n_2 - n_1}{2}\cdot
$$
Then Allaart's inequality applied to $m, \ell$ gives
$$
S_2(n_2) + S_2(n_1) - 2S_2\left(\frac{n_1 + n_2}{2}\right) \leq 
\frac{n_2 - n_1}{2}\cdot
$$
Now, this inequality after applying Equality~\eqref{eq:S(2t)} with
$t = \frac{n_1 + n_2}{2}$, can be written as
$$
S(n_2) + S_2(n_1) - S_2(n_1+n_2) + \frac{n_1 + n_2}{2}
\leq \frac{n_2 - n_1}{2}, 
$$
and hence
$$
S_2(n_2) + S_2(n_1) + n_1 \leq S_2(n_1+n_2). 
$$
It seems that Allaart's inequality for $p=0$, in which both 
$m + \ell$ and $m - \ell$ (necessarily of the same parity) 
occur, does not immediately imply Graham's, where $n_1$ and 
$n_2$ may have opposite parities. This suggests the possible 
existence of a ``Graham-Allaart'' inequality. 
\end{remark}

\section{A variation on Theorem~\ref{th:bio} and two  
generalizations}\label{sec:results}

In this section we state and prove three results inspired by Theorem~\ref{th:bio}, from which, together with Theorem~\ref{th:bio} itself, most of the results that we 
found in the literature can be deduced (except Allaart's in~\cite{Alla11}, i.e., Theorem~\ref{th:allaart}, for $p \neq 0$, 
and a sharp inequality due to Allaart~\cite{Alla14}, see Remark~\ref{rk:allaart-sharp}).

\subsection{A variation on Theorem~\ref{th:bio}}
We begin with a variation of~\cite[Thm~5.1, pp.~12--13]{bio} (stated as Theorem~\ref{th:bio} above).

\begin{theorem}\label{th:variation-bio}
Let $b$ be an integer $\geq 2$. Let $k_1 \leq k_2 \leq \dots \leq k_b$ 
be nonnegative integers. Then 
$$
S_b(k_1 + k _2 + \dots + k_b)
+ \sum_{j=1}^{b-1} S_b(k_b - k_j) 
-  b \, S_b(k_b) \leq \sum_{i=1}^{b-1} i \, k_i.
$$
\end{theorem}

\begin{proof}
Define the integers $n_i$ by $n_i := k_b - k_{b-i}$, for all
$i \in  [1, b-1]$, and $n_b := \sum_{1 \leq i \leq b} k_i$.
Since $0 \leq n_1 \leq n_2 \leq \dots \leq n_{b-1} \leq n_b$
we can apply Inequality~\eqref{ineq:bio-base-b}. Thus we obtain
$$
S_b(k_1 + k _2 + \dots + k_b)
+ \sum_{j=1}^{b-1} S_b(k_b - k_{b-j})
+ \sum_{i=1}^{b-1} (b-i) (k_b - k_{b-i}) \leq S_b(b k_b).
$$
But $S_b(b k_b) = b \, S_b(k_b) + \frac{b(b-1)}{2}\, k_b$
(see~\cite{Alla14} or see Lemma~\ref{lemma:allaart} above). Hence
$$
S_b(k_1 + k _2 + \dots + k_b)
+ \sum_{j=1}^{b-1} S_b(k_b - k_{b-j}) - b \,S_b(k_b)
\leq \sum_{i=1}^{b-1} (b-i) \, k_{b-i}
$$
i.e.,
$$
S_b(k_1 + k _2 + \dots + k_b)
+ \sum_{j=1}^{b-1} S_b(k_b - k_j) - b \, S_b(k_b)
\leq \sum_{i=1}^{b-1} i \, k_i.
$$
This finishes the proof.
\end{proof}

\subsection{Generalizations of Theorems~\ref{th:bio} and~\ref{th:variation-bio}}

In~\cite[Thm~5.1]{bio} (see Theorem~\ref{th:bio} above)
we can drop the hypothesis that the number of $n_i$ is equal
to the base $b$ and replace it with the assumption that the
number of $n_i$ is at most equal to the base $b$.

\begin{theorem}\label{th:slight}
Let $b$ be an integer $\geq 2$. For all integers $n\ge 0$, let $s_b(n)$ denote the sum of  
digits in the base-$b$ expansion of the integer $n$, and define
$S_b(n) := \sum_{1 \leq j \leq n-1} s_b(j)$. Let $r$ be an 
integer in $[1, b]$. Let $n_1, n_2, \dots, n_r$ be integers 
such that $0 \leq n_1 \leq n_2 \leq \dots \leq n_r$. Then the following 
equality holds:
\begin{equation}\label{ineq:slight-base-b}
\sum_{i=1}^{r} S_b(n_i) + 
\sum_{i=1}^{r-1} (r-i)n_i
\leq S_b \left(\sum_{i=1}^{r} n_i\right).
\end{equation}
\end{theorem}

\begin{proof}
In Inequality~\eqref{ineq:bio-base-b} of Theorem~\ref{th:bio}, take
$n_1 = n_2 = \cdots = n_{b-r} = 0$. This gives 
$$
\sum_{i=b-r+1}^{b} S_b(n_i) + 
\sum_{i=b-r+1}^{b-1} (b-i) n_i \leq 
S_b\left(\sum_{i=b-r+1}^{b} n_i\right),
$$
and hence, with the change of indices $b-r+j = i$, we get
$$
\sum_{j=1}^{r} S_b(n_{b-r+j}) + 
\sum_{j=1}^{r-1} (r-j) n_{b-r+j} \leq 
S_b\left(\sum_{j=1}^{r} n_{b-r+j}\right).
$$
Now, define $m_j := n_{b-r+j}$ for all $j \in [1,r]$. Then  
$0 \leq m_1 \leq m_2 \leq \dots \leq m_r$
and
$$
\sum_{j=1}^{r} S_b(m_j) + 
\sum_{j=1}^{r-1} (r-j) m_j
\leq S_b\left(\sum_{i=1}^{r} m_j\right).
$$
This ends the proof.
\end{proof}

In the same spirit one can extend 
Theorem~\ref{th:variation-bio}.

\begin{theorem}\label{th:slight-2}
Let $b$ be an integer $\geq 2$ and let $r$ be an integer in $[1, b]$. Let 
$m_1 \leq \dots \leq m_r$ be non-negative integers. Then
\begin{equation}\label{ineq:slight-2}
S_b(m_1 + m_2 + \dots + m_r) 
+ \sum_{j=1}^{r-1} S_b(m_r - m_j) - r S_b(m_r) \leq
\sum_{j=1}^{r-1} (b - r + j) m_j.
\end{equation}
\end{theorem}

\begin{proof}
Let $k_1, k_2, \ldots, k_b$ be integers such that 
$k_1 = k_2 = \dots = k_{b-r} :=0$ and also satisfying 
$0 \leq k_{b-r+1} \leq k_{b-r+2} \leq \dots \leq k_b$.
Applying Theorem~\ref{th:variation-bio} yields
\begin{equation*} 
S_b(k_{b-r+1} + \dots + k_b) 
+ (b-r) S_b(k_b) + \sum_{j=b-r+1}^{b-1} S_b(k_b - k_j)
- b \, S_b(k_b) 
\leq \sum_{jb-r+1}^{b-1} j\, k_j.
\end{equation*}
Changing the indices in the last two sums gives
\begin{equation*} 
S_b(k_{b-r+1} + \dots + k_b) 
+ (b-r) S_b(k_b) + \sum_{j=1}^{r-1} S_b(k_b - k_{b-r+j})
- b \, S_b(k_b)  
\leq \sum_{j=1}^{r-1} (b-r+j) k_{b-r+j},
\end{equation*}
and hence, by grouping the terms in $S_b(k_b)$ and letting 
$m_j := k_{b-r+j}$,
\begin{equation*} 
S_b(m_1 + \dots + m_r) 
- r \, S_b(m_r) + \sum_{j=1}^{r-1} S_b(m_r-m_j)
\leq \sum_{j=1}^{r-1} (b-r+j) m_j.
\end{equation*}
as desired.
\end{proof}

\subsection{(Non-)Optimality in the theorems of this section}
 
One can ask, e.g., whether Theorem~\ref{th:slight} can be further
generalized by taking an integer $r > b$. The answer is no: one 
can show that Theorem~\ref{th:slight} is optimal in the 
sense that for any integer $r > b$ there do not exist constants 
$\alpha_j \geq 1$ such that for all integers $n_1, n_2, \dots, n_r$ one 
has the inequality
$$
\sum_{i=1}^{r} S_b(n_i) + 
\sum_{i=1}^{r-1} \alpha_i n_i
\leq S_b \left(\sum_{i=1}^{r} n_i\right).
$$
Namely, we prove the following theorem.

\begin{theorem}\label{b+1}
Let $b$ be an integer $\geq 2$ and let $r$ be an integer $\geq b + 1$. 
Then there exist integers $n_1, n_2, \dots, n_r$ with 
$0 \leq n_1 \leq \dots \leq n_r$ such that
$$
\sum_{i=1}^{r} S_b(n_i) + \sum_{i=1}^{r} n_i > 
S_b\left(\sum_{i=1}^{r} n_i\right). 
$$
\end{theorem}

\begin{proof}
Take $n_1 = n_2 = \dots = n_{r-b-1} := 0$, 
$n_{r-b} = n_{r-b+1} = \dots = n_{r-1} := 1$,
and $n_r = b^x$ where $x$ is an integer $\geq 2$. 
Then, on the one hand, 
\begin{align*}
\sum_{i=1}^{r} S_b(n_i) + 
\sum_{i=1}^{r} n_i
&= \sum_{i=r-b}^{r} S_b(n_i) + b + b^x \\
&= b S_b(1) + S_b(b^x) + b + b^x = S_b(b^x) + b + b^x
\end{align*}
and on the other, 
\begin{align*}
    S_b\left(\sum_{i=1}^{r} n_i\right)
&= S_b\left(\sum_{i=r-b}^{r} n_i \right)
= S_b(b + b^x) \\
&= S_b(b^x) + s_b(b^x) + s_b(b^x + 1) + \dots + s_b(b^x + b - 1)\\
&= S_b(b^x) + 1 + 2 + \dots + b \\
&= S_b(b^x) + \dfrac{b(b+1)}{2} < S_b(b^x) + b + b^x,
\end{align*}
where we use the fact that $x\geq 2$.
This finishes up the proof.
\end{proof}

\begin{remark}
The right-hand term of Inequality~\eqref{ineq:slight-2} in Theorem~\ref{th:slight-2} is not optimal: e.g., take $b \geq 4$,
$r=2$, and see Remark~\ref{rk:allaart-sharp} below.
\end{remark}

\section[Graham's inequality and its first generalizations]{Graham's inequality and its first generalizations by
Allaart and Cooper are consequences of Theorem~\ref{th:bio}}

Graham's theorem was given above as Theorem~\ref{th:graham}.
The following generalization for any base $b \geq 2$ and two 
integers $n_1, n_2$, was proved by Allaart in~\cite{Alla14} 
and again quite recently by Cooper~\cite{Coo23}.

\begin{theorem}[in~\cite{Alla14}]
Let $b \ge 2$ be an integer. For all integers $n_1,n_2$ with 
$0<n_1 \le n_2$, we have 
\begin{equation}\label{ineq:allaart-cooper-base-b}
    S_b(n_1) + S_b(n_2) + n_1 \leq S_b(n_1+n_2).
\end{equation}
\end{theorem}

It is immediate that this statement is implied by Theorem~\ref{th:slight} by taking $r=2$. Hence so is 
Graham's result by taking $b=r=2$.

Another result is proved in~\cite{Alla14}, namely:

\begin{theorem}[in~\cite{Alla14}]\label{th2-allaart-arxiv}
For any integers $k, \ell$ and $m$ with 
$0 \leq \ell \leq k \leq m$, 
we have
\begin{equation}\label{ineq:base-3}
S_3(m + k + \ell) + S_3(m - k) + S_3(m - \ell) - 3S_3(m) 
\leq 2k + \ell.
\end{equation}

\end{theorem}

This theorem is an easy consequence of our Theorem~\ref{th:variation-bio} 
(and hence of~\cite[Thm.~5.1]{bio}, see Theorem~\ref{th:bio} above): 
indeed, take $b=3$.

%
%

\begin{remark}\label{rk:allaart-sharp}
On~\cite[p.~680]{Alla14}, Allaart notes that, 
by taking $\ell = 0$, Inequality~\eqref{ineq:base-3} gives: 
for all integers $k,m$ with $0 \leq k \leq m$, one obtains
\begin{equation*}
S_3(m+k) + S_3(m-k) - 2S_3(m) \leq 2k.
\end{equation*}
Then, Allaart proves the following (sharp) inequality
in~\cite[Thm.~3, p.~681]{Alla14} : 

\begin{theorem}[in~\cite{Alla14}]
Let $b$ be an integer $\ge 2$.
For all integers $k,m$ with $0 \leq k \leq m$, we have
\begin{equation}\label{ineq:strong}
S_b(m+k) + S_b(m-k) - 2S_b(m) 
\leq \left\lfloor \frac{b+1}{2} \right\rfloor k.
\end{equation}
\end{theorem}
For $b \geq 4$ this inequality is stronger than Inequality~\eqref{ineq:slight-2} for $r=2$, which only gives
$$
S_b(m+k) + S_b(m-k) - 2S_b(m) \leq (b-1) k.
$$
We did not succeed in deducing Inequality~\eqref{ineq:strong} from
the result of~\cite{bio} or variations thereof.

\end{remark}
\section{A binomial digression}\label{sec:binom}

An easy inequality mentioned on~\cite[p.~682]{Alla14} reads:
for any nonnegative integers $n, k$ we have
$s_b(n+b^k) \leq s_b(n)+1$. A more general, probably well-known,
inequality, is that for any nonnegative integers $n, m$, we have
$s_b(n+m) \leq s_b(n) + s_b(m)$ (see, e.g.,~\cite[Prop.~2.1]{HLS}).
A way of proving this inequality {\em when $b$ is prime}, is to 
use a result of Legendre: 
$\nu_b(n!) = \frac{n - s_b(n)}{b-1}$, where $\nu_b(k)$ is the 
$b$-adic valuation of the positive integer $k$ (see ~\cite[pp.~10--12]{Legendre}). This implies 
easily $s_b(m) + s_b(n) - s_b(n+m) = \nu_b(\binom{n+m}{n})$.
Since $\nu_b(\binom{n+m}{n}) \geq 0$, we are done. 
This inequality raises the question of whether something
similar (at least when $b$ is prime) could be done for, 
say, Theorem~\ref{th:bio} and/or Theorem~\ref{th:allaart} above.
For the second one, we note that it might be necessary 
to introduce a kind of generalized binomial coefficient. 

\section{How to generalize Allaart's Theorem~\ref{th:allaart}?}
\label{sec:trial}

It is tempting to try to generalize Theorem~\ref{th:allaart}. 
A reasonable idea seems to replace the sequence
$(2^{pi})_{i \geq 0}$ with a sequence $(\lambda_i)_{i \geq 0}$ that is well chosen.
This leads to the following definition.

\begin{definition}
\label{lambda}
Let $(\lambda_i)_{i\ge 0}$ be a sequence of positive real numbers.
For all integers $n\ge 0$, if we let $(d_i)_{i\ge 0}$ be the binary digits of $n$, we define 
$\wei_{(\lambda)}(n)=\sum_{i=0}^{+\infty} \lambda_i d_i$ and its 
summatory function $\Wei_{(\lambda)}(n)=\sum_{m=0}^{n-1} 
\wei_{(\lambda)}(m)$.
\end{definition}

\begin{example}
If $\lambda_i := 2^{pi}$, with $p \in [0,1]$, then 
$\wei_{(\lambda)}$ and $\Wei_{(\lambda)}$ are exactly
the quantities $\wei_p$ and $\Wei_p$ in Theorem~\ref{th:allaart}
above.
\end{example}

In our quest for finding other sequences $(\lambda_i)_{i\ge 0}$
for which an analog of Theorem~\ref{th:allaart} would hold, we first tried to impose conditions like: 
{\it $(\lambda_i)_{i\ge 0}$ is non-decreasing, but not ``too 
much''}. For instance, we tried to impose:
$\lambda_i \leq \lambda_{i+1} \leq C \lambda_i$
for some constant $C \geq 2$ and for all $i$. 
However, this does not work. Namely, take
$\lambda_i := 3^i$, then 
\[
\sum_{i=0}^{+\infty} 3^{i} d_i 
= \sum_{i=0}^{+\infty} 2^{\log_2(3) i} d_i
= \sum_{i=0}^{+\infty} 2^{qi} d_i
\]
with $q = \log_2(3) > 1$. The remark below shows that 
Allaart's Inequality~(\ref{eq:Allaart-2powerp}) in
Theorem~\ref{th:allaart} is not true for this sequence 
$(\lambda_i)_{i\ge 0} = (3^i)_{i\ge 0}$.

\begin{remark} In the hypotheses of 
Theorem~\ref{th:allaart}, let us replace $p \in [0,1]$ 
with some $p > 1$ (e.g., $p = \log_2(3)$), 
and take $\ell = 1$. If Allaart's inequality 
were true, we would have
$$
\Wei_p(m+1) + \Wei_p(m-1) - 2 \Wei_p(m) \leq 1,
$$
which is equivalent to saying that
\begin{align}
\label{eq:p strict 1}
    \wei_p(m) - \wei_p(m-1) \leq 1.
\end{align}
Now, let $m$ be a power of $2$, say $2^k$ with $k$ large. The
binary expansion of $m$ is $10^k$ and that of
$(m-1)$ is $1^k$ (where, for $a\in \{0,1\}$, $a^k$ means that the digit
$a$ is repeated $k$ times). Therefore
$$
\wei_p(m) -\wei_p(m-1) = 2^{pk} \cdot 1 
- \sum_{i=0}^{k-1} 2^{pi} \cdot 1 
= \frac{2^{pk}(2^p-1) - 2^{pk} + 1}{2^p - 1},
$$
which behaves like $\frac{2^{pk}(2^p-1)}{2^p - 1} = 2^{pk}$ when 
$k$ goes to infinity (recall that $p > 1$, and hence $2^p-1 > 1$).
This contradicts Inequality~\eqref{eq:p strict 1}.
\end{remark}

\section{Questions and expectations}

We propose the following questions or/and expectations.

\begin{itemize}

\item[*] Generalize Theorem~\ref{th:allaart}: is there a 
generalized Inequality~\eqref{eq:Allaart-2powerp} and/or a 
generalized Inequality~\eqref{ineq:bio-base-b}? In doing so, recall Section~\ref{sec:trial} above.

\item[*] Give a proof of Inequality~\eqref{ineq:bio-base-b} or even of Inequality~\eqref{ineq:slight-base-b} using the method of~\cite{Allo21}.

\item[*] Is there a ``Graham-Allaart inequality''? See the end of Section~\ref{sec:Graham-Allaart}.

\item[*] To what extent is it possible to address inequalities
mentioned in this paper, through the use of (generalized)
binomial coefficients? See the end of Section~\ref{sec:binom}.

\item[*] Are there similar inequalities if the sum of digits
is replaced with another ``block counting-function'' (e.g., the 
number of $11$ in the binary expansion of the integer $n$)?
It is possible that the papers~\cite{Prodinger,Kirschenhofer,MOSS} 
yield some hints in this direction.
\end{itemize}

\subsection*{Acknowledgments}
Jean-Paul Allouche wants to thank Alain Denise for discussions  about the biology part of~\cite{bio}.
Both authors thank Michel Dekking, Jeff Lagarias, and Jeff Shallit for their comments, and Jia-Yan Yao and Pieter Allaart for providing the authors with~\cite{Alla14}.

Manon Stipulanti is an FNRS Research Associate supported by the Research grant 
1.C.104.24F.


\EditInfo{November 29, 2023}{April 25, 2024}{Emilie Charlier, Julien Leroy and Michel Rigo}

\end{document}